%% file: FelipeTeissier.tex
\documentclass[12pt,twoside,a4paper]{article}
\usepackage[utf8]{inputenc}
\usepackage{amsfonts}
\usepackage{amssymb}
\usepackage{amsthm}
\usepackage[bookmarksopen=false,
breaklinks=true,%
      backref=page,pagebackref=true,plainpages=false,%
      hyperindex=true,pdfstartview=FitH,%
      pdfpagelabels=true,colorlinks=true,linkcolor=blue,%
      citecolor=red,urlcolor=red,hypertexnames=false,colorlinks=false%
      ]%
   {hyperref}
 
\usepackage{mathrsfs}
\usepackage{mathtools}
\usepackage{stmaryrd}
\usepackage[all]{xy}
\usepackage{makeidx}
\usepackage[french]{babel}
\usepackage{xspace}
\usepackage{raggsect}

\usepackage{etoolbox}
\makeatletter\def\th@plain{\slshape}
\patchcmd{\th@remark}{\itshape}{\slshape}{}{}\makeatother

\newcounter{bidon}

\usepackage{stixgras}
\RequirePackage{amsbsy}

\newcommand \he{^{\mathrm{h}}}
\newcommand \He{^{\mathrm{H}}}

\input{FrenchTheoremsBT.tex}

\input{FrenchMacrosBT.tex}

\input{MathMacrosBT.tex}

\DeclareMathAlphabet{\mathpzc}{OT1}{pzc}{m}{it}

\patchcmd{\sectionmark}{\MakeUppercase}{}{}{}
\begin{document}
\stMF

\title{À propos d'un théorème de de Felipe et Teissier sur la comparaison de deux hensélisés
dans le cas non~noethérien}

\author{M.-E. Alonso\thanks{Universitad Complutense, Madrid, Espa\~na. {\tt M\(_-\)Alonso@Mat.UCM.Es}}
 \and  Henri Lombardi
\thanks{Univ. de Franche-Comt\'e, 25030
Besan\c{c}on cedex, France. {\tt henri.lombardi@univ-fcomte.fr}}
\and
Stefan Neuwirth
\thanks{Univ. de Franche-Comt\'e, 25030
Besan\c{c}on cedex, France. {\tt stefan.neuwirth@univ-fcomte.fr}}}
\maketitle


\tableofcontents

\begin{abstract} 
This paper gives an elementary proof of a theorem by de Felipe and Teissier in the paper ``Valuations and henselization'' (arxiv.org/abs/1903.10793v1), to appear in Math. Annalen. The theorem compares two henselizations of a local domain dominated by a valuation domain. Our proofs are written in the constructive Bishop style.
\end{abstract}


\section*{Introduction}
\addcontentsline{toc}{section}{Introduction}
Cet article est écrit dans le style des \coma à la Bishop (\cite{Bi67,BB85,BR1987,MemoireFFR,CACM,MRR,Yen2015}).
Il peut être vu comme la continuation naturelle des articles \cite{ALP08,CLR2001,KL00,KLP03}.

Le \tho auquel le titre réfère se trouve dans la prépublication [BT2019, \textsc{de Felipe} et \textsc{Teissier}, \textsl{Valuations and henselization}, 2012, \url{https://arxiv.org/abs/1903.10793v1}] écrit en prolongement de \cite{HOST2012}. L'article a été accepté à Math. Annalen.
Le \tho est le suivant.


\begin{thm} \label{thFT}
Let $R$ be a local domain and let $R\he$ be its henselization. If $v$ is a valuation centered in $R$, then:
\begin{enumerate}
\item There exists a unique prime ideal $H(v)$ of $R\he$ lying over the zero ideal of $R$ such that $v$  extends to a valuation $v'$ centered in $R\he/H(v)$ through the inclusion $R\subseteq  R\he/H(v)$. In addition, the ideal $H(v)$ is a minimal prime and the extension $v'$ is unique.
\item With the notation of 1., the valuations $v$ and $v'$ have the same value group.
\end{enumerate}
 \end{thm}

\medskip La technique de démonstration dans [BT2019] est assez sophistiquée et utilise des méthodes topologiques\footnote{Il y est question de suite pseudo convergentes dans une complétion ad hoc.}. Mais comme la construction du hensélisé d'un anneau local résiduellement discret et celle du hensélisé d'un domaine de valuation relèvent de méthodes \covs purement \agqs (voir \cite{ALP08} et \cite{KL00}), 
on peut à priori espérer une démonstration \cov purement \agq du résultat.

C'est ce que nous réalisons dans cet article.

\smallskip Dans la section 1, nous donnons le contexte constructif du \tho.
C'est évidemment le même que le contexte en \clama mais nous donnons quelques précisions \ncrs pour que les résultats aient tous une signification \algq.
Par exemple pour construire le hensélisé d'un anneau local, nous avons besoin que cet anneau soit local au sens \cof et que le corps résiduel possède un test d'\egt.

Nous indiquons sommairement dans cette section les \thos \cofs précédemment établis concernant le hensélisé d'un anneau local et le hensélisé d'un corps valué.

\smallskip Nous notons que le hensélisé d'un corps valué est unique, donc le \thref{thFT} est en fait un \tho qui compare deux hensélisés. Celui de l'anneau local intègre de départ et celui de l'anneau de valuation qui domine  cette annau local intègre. Le point \textsl{1} du \tho dit que le morphisme local naturel du premier hensélisé dans le second a pour noyau un \idep minimal. Le point \emph{2} est selon nous un simple rappel de ce que lorsqu'on étend un \adv à son hénsélisé, le groupe de valeurs ne change pas (ni d'ailleurs le corps résiduel).

Pour démontrer le \thref{thFT}, il suffit de le faire en remplaçant les hensélisés par des étages finis de leurs constructions simultanées, c'est ce que nous faisons dans les sections suivantes.

\smallskip La section 2 traite le cas d'un anneau local intègre avec le \thref{thFeTe1}. Ce cas est à priori suffisant pour conclure mais cela fait appel à un \tho dont nous ne connaissons pas de \demo simple: tout étage fini de la construction d'un hensélisé peut être réalisé en une seule étape. D'où la section 3.

\smallskip  La section 3 traite le cas d'un anneau local réduit minimalement valué avec le \thref{thFeTe2}. Ce cas permet de suivre la construction des hensélisés étage par étage.
En outre la section 2 n'est plus vraiment \ncr.

\smallskip Enfin une annexe en section 4 traite pour le fun le cas d'un anneau local minimalement valué qui n'est pas supposé réduit avec le \thref{thFeTe2}. Ceci nous donne donc la \gnn suivante du \thref{thFT}.

\begin{thm} \label{thFT2}
Let $R$ be a local ring, $V$ a valuation domain and $\varphi:R\to V$ a local morphism whose kernel is a minimal prime. Let $R\he$ the henselization of $R$ as a local ring and $V\He$ the henselization of $V$ as a valuation domain. Then the kernel of the canonical local morphism $\psi:R\he\to V\He$ is a minimal prime.
 \end{thm}

\section{Le contexte} 
Nous fixons dans cette section la terminologie \cov usuelle et nous précisons le contexte  qui permet la construction de hensélisés en \coma.

On note $\Ati$  le groupe des unités de l'anneau commutatif \(\gA\).
Le radical de Jacobson $\Rad(\gA)$ de  \(\gA\) est l'idéal formé par les \elts 
\(x\)  qui satisfont $1+x\gA\subseteq \Ati$.   

Un \textsl{corps discret} est un anneau non trivial dans lequel tout \elt est nul ou inversible, ceci de manière explicite (il y a en particulier un test à $0$).

Si $\gA$ est non trivial et si tout \elt de \(\gA\) est nul ou régulier on dit que \(\gA\) est un anneau \textsl{intègre}. Un anneau est dit \textsl{sans diviseur de zéro} lorsque $xy=0$ implique $x=0$ ou $y=0$. Cette notion est \cot un peu plus faible que la notion d'anneau intègre. Un anneau non trivial \sdz est intègre \ssi il possède un test à zéro, \cad si l'égalité est décidable. Un anneau \sdz est \textsl{réduit}: tout \elt nilpotent est nul.

L'\textsl{anneau total de fractions} d'un anneau \(\gA\), noté $\Frac\gA$, est obtenu en forçant l'inversibilité des \elts réguliers. Si \(\gA\) est intègre, 
$\Frac\gA$ est un \cdi, appelé le \textsl{corps de fractions de $\gA$}. Par ailleurs, tout sous-anneau d'un \cdi est intègre.

Un anneau est \textsl{normal} s'il satisfait l'axiome suivant\footnote{Tout \idp est \icl.}:
$$ x^r=a_1bx^{r-1}+a_2b^2 x^{r-2}+\cdots +a_rb^r\Rightarrow \exists y\;\; x=by.$$
Un anneau normal intègre est appelé un anneau \textsl{\icl}. Il est \icl dans son corps de fraction: tout \elt de $\Frac\gA$ entier sur \(\gA\) est dans \(\gA\)\footnote{En \clama un anneau est normal \ssi tout localisé en un \idep est \icl.}.

Un anneau $\gA$ est dit \textsl{\zed} si pour tout \(x\) il existe un \(y\) et un entier \(n\geq 0\) tels que \(x^n(1-xy)=0\). Un anneau \zed non trivial et \sdz est un \cdi.   

\subsubsection*{Anneau local hensélien} 
Un \textsl{anneau local} \(\gA\) est un anneau commutatif unitaire non trivial dans lequel est satisfait l'axiome
$$ x+y\in \Ati \Rightarrow x \in \Ati \hbox{ ou }y \in \Ati.
$$
Ici le \gui{ou} a son sens \cof explicite.
Un morphisme d'anneaux $\gA\to\gB$ entre anneaux locaux est appelé un \textsl{morphisme local} s'il réfléchit les unités, \cad si \hbox{$\varphi(x)\in\gB\eti$} implique $x\in\Ati$. 

Pour un anneau local, on note souvent $\fm_\gA=\Rad(\gA)$; c'est un \idep maximal\footnote{Mais cela ne peut pas être pris comme \dfn \cov du radical de Jacobson.}.

Un anneau local \(\gA\) possède un \textsl{corps résiduel} \(\kappa_\gA=\gA/\fm_\gA\)\footnote{En \coma on appelle \textsl{corps de Heyting}, ou plus simplement \textsl{corps} un anneau local dont le radical de Jacobson est nul.}. 
Le morphisme naturel $\gA\to \kappa_\gA$ est local. Un quotient d'un \alo est un \alo de même corps résiduel et le morphisme canonique sur le quotient est un morphisme local.

Un anneau local est dit \textsl{résiduellement discret} si son corps résiduel est discret. Cela revient à dire que pour tout \(x\in\gA \), on a \gui{ \(x\in\Rad(\gA)\) ou \( x\in\Ati\)}
de manière explicite.

Un anneau est \textsl{connexe} si tout \idm est égal à $0$ ou $1$. Les anneaux intègres et les anneaux locaux sont connexes. Un anneau est local \zed \ssi tout \elt est inversible ou nilpotent. Un anneau non trivial \zedr et connexe est un \cdi.

\begin{definition} \label{defi1.1} 
Soit \((\gA,\fm_\gA)\) un anneau local. Nous disons que \(\gA\) est {\em hensélien} si tout \pol unitaire
\(f(X)=X^n+ \cdots +a_1 X+a_0\in A[X]\) avec \(a_1\in \Ati\) et 
\(a_0\in\fm_\gA\) a une racine dans \(\fm_\gA\).
\end{definition}

Un \pol tel que $f(X)$ dans la \dfn ci-dessus est appelé un \textsl{\pol de Nagata}. Plus \gnlt nous appellerons \textsl{code de Hensel} un couple $(f,a)$ où $f\in \AX$, $a\in\gA$, $f(a)\in\fm_\gA$ et $f'(a)\in\Ati$.
Ainsi $(f,0)$ est un code de Hensel si \(f\) est un \pol de Nagata.
Dans un anneau local hensélien, si $(f,a)$ est un code de Hensel, il existe un unique $\alpha\in\gA$ tel que $f(\alpha)=0$ et $\alpha-a\in\fm_\gA$.

Par \dfn, un \textsl{hensélisé} de \((\gA,\fm_\gA)\) est donné par un anneau local hensélien \((\gA\he,\fm_{\gA\he})\) et un morphisme local $\varphi\he:\gA\to\gA\he$ qui factorise de manière unique tout morphisme local de \(\gA\) vers un anneau local hensélien.
S'il existe, le hensélisé est unique à isomorphisme local unique près.

Lorsqu'un anneau local \(\gA\) est résiduellement discret\footnote{Rappelons qu'en \clama tout anneau local est résiduellement discret.}, il possède en \coma\footnote{Donc en \clama tout anneau local possède un hensélisé, \gui{construit} de la même manière que ce que l'on fait en \coma, mais en utilisant le principe du tiers exclu pour décider l'\egt dans le corps résiduel.} un hensélisé qui peut être construit par étapes. Une étape \elr consiste à ajouter un zéro à la Hensel de manière optimale (\cad que l'extension doit satisfaire la \prt universelle adéquate).

\begin{definition} \label{defi1.2}
Soit \(f(X)=X^n+ \cdots + a_1 X+a_0\in \AX\) un \pol de Nagata. Nous notons
 \(\gA_f\) l'anneau défini de la manière suivante: soit \(\gB=\gA[x]=\aqo{\AX}{f(X)}\)
(où \(x\) est la classe de \(X\)) et soit
\(U_f\subseteq \gB\) le \mo de \(\gB\) défini par
\[U_f=\sotQ{ g(x)\in \gB}  {g(X)\in \gA[X],\ \ov{g(0)}\in (\kappa_\gA)\eti} .
\]
Alors \(\gA_f=U_f^{-1}\gB\). 
\end{definition}

\begin{theorem}[\cite{ALP08}] \label{DeAaAf} Soit \((\gA,\fm)\) un \alo \dcd et \(f(X)\in \AX\) un \pol de Nagata.
\begin{enumerate}
\item L'anneau \(\gA_f\) est un anneau local résiduellement discret.
\item  Son \idema est \(\fm\cdot \gA_f\).
\item Son corps résiduel est (canoniqueent isomorphe à) \(\kappa_\gA\).
\item L'anneau \(\gA_f\) est une \alg \fpte sur \(\gA\). En particulier on peut identifier \(\gA\) avec son image dans \(\gA_f\), et écrire \(\gA\subseteq \gA_f\).
\item Si $\varphi:\gA\to\gC$ est un morphisme local d'anneaux locaux résiduellement discrets et si $\varphi(f)$ a un zéro dans $\fm_\gC$, 
il existe un unique morphisme local \(\gA_f\to \gC\) qui factorise $\varphi$. 
\end{enumerate}
\end{theorem}

Le hensélisé \(\gA\he\) est construit comme colimite filtrée de la famille des extensions du type  \((\dots((\gA_{f_1})_{f_2})\dots)_{f_m}\) pour des \pols de Nagata \(f_j\) successifs. Cette colimite est bien définie en raison de la \prt \uvle rappelée dans le point \textsl{5}.

\begin{definition}\label{def poly spe} Soit  \((\gA,\fm_\gA)\)  un \alo \dcd. 
Un polynôme  \(t\in \AX\)  est dit {\em  spécial} s'il est de la forme 
\(X^d-X^{d-1}+t_{d-2}X^{d-2}+\cdots +t_1X+t_0\) avec les  \(~t_i\in 
\fm_\gA\).  Le zéro \(\alpha\) de \(t\) dans \(\gA\he\) codé à la Hensel par \((t,1)\) est appelé  le {\em  zéro spécial} de~\(t\). Le \pol \(g(X)=t(X+1)\)
est un \pol de Nagata.
\end{definition}

Dans \cite[Lemma 5.3]{ALP08} il est démontré que le \pol de Nagata \(f\) du \thref{DeAaAf}, ou tout autre \pol de \(f\in\AX\) vérifiant \(f(0)\in\fm_\gA\) et \(f'(0)\in\Ati\), peut être remplacé par un \pol de Nagata de la forme \(g(X)=t(X+1)\) où \(t(X)\) est un \pol spécial, avec \(\gA_f\) canoniquement isomorphe à \(\gA_g\).

\begin{lemma} \label{lemHensNormal}
Si \(\gA\) est normal, \ivmp \(\gA_f\).
Si \(\gA\) est réduit, \ivmp \(\gA_f\).  
\end{lemma}
%
\begin{proof}
Les \demos dans \cite{Ray70} sont \covs pour l'essentiel. Pour plus de précisions voir \cite{CL2016b}.
\end{proof}

Par contre, lorsque \(\gA\) est seulement supposé intègre, \(\gA_f\) est en général seulement réduit.

\subsubsection*{Corps valué hensélien} 

Un sous-anneau $\gV$ d'un corps discret \(\gK\) est appelé un \textsl{anneau de valuation de \(\gK\)} si pour tout \(x\in\gK\eti\), on a $x\in\gV$ ou $x^{-1}\in\gV$ de manière explicite. Alors le corps \(\gK\)  est le corps de fractions de $\gV$ et $\gV$  est local et \icl. 

Un \textsl{corps valué (discret)} est un couple $(\gK,\gV)$ où \(\gK\) est un corps discret et $\gV$ un \adv de \(\gK\).
Dans la suite,  \textsl{on suppose toujours que $\gV$ est résiduellement discret}. Cela revient à dire que l'on a un test de divisibilité entre \elts de $\gV$.
Le groupe $\Gamma=\gK\eti\!/\gV\eti$ est alors un groupe totalement ordonné discret dont les \elts $\geq 0$ sont les (classes des) \elts de $\gV\setminus\so0$.
On définit la \textsl{valuation} \hbox{$v:\gK\to\Gamma_\infty=\Gamma\cup\so\infty$} du corps valué comme 
étant égale au morphisme canonique de \(\gK\eti\)  sur $\Gamma$ que l'on prolonge à~\(\gK\)  en posant $v(0)=\infty$.
On a alors $\gV=\sotq{x\in\gK}{v(x)\geq 0}$ et $\fm_\gV=\sotq{x\in\gK}{v(x)> 0}$. 

Un anneau intègre est appelé un \textsl{domaine de valuation} si c'est un anneau de valuation de son corps de fractions. 

On peut aussi définir un anneau de valuation intègre comme un triplet $(\gV,\Gamma,v)$  où \(\gV\)  est un anneau intègre, $\Gamma$ est un groupe totalement ordonné discret et $v:\gV\to \Gamma_\infty$ est une \textsl{valuation}, \cad une application vérifiant
\begin{itemize}
\item $v(x)\geq 0$ pour tout $x\in\gV$,
\item  $v(ab)=v(a)+v(b)$,
\item  $v(a+b)\geq \min(v(a),v(b))$, et
\item $x=0 \Leftrightarrow v(x)=\infty$. 
\end{itemize}
Si \(\gK\) est le corps de fractions de $\gV$, la valuation s'étend de manière unique à \(\gK\) et le groupe $\gK\eti\!/\gV\eti$ s'identifie via $v$  à un sous-groupe de $\Gamma$.

Une \emph{extension} du corps valué $(\gK,\gV)$ est un corps valué $(\gL,\gW)$ où $\gK\subseteq \gL$ et \hbox{$\gV=\gK\cap\gW$}\footnote{On devrait plutôt considérer un morphisme $\varphi:\gK\to\gL$ avec $\gV=\gK\cap\varphi^{-1}(\gW)$. On est alors dans la situation voulue à un isomorphisme unique près.}. Le groupe $\gK\eti\!/\gV\eti$ s'identifie alors à un sous-groupe de $\gL\eti\!/\gW\eti$ et le corps résiduel $\gV/\fm_\gV$ à un sous-corps de $\gW/\fm_\gW$. Et la valuation $v_\gK$ se prolonge en la valuation~$v_\gL$. 

Un corps valué $(\gK,\gV)$ est dit \textsl{hensélien} si $\gV$ est hensélien en tant qu'anneau local. 

Par \dfn, un \textsl{hensélisé} de \((\gK,\gV)\) est donné par une extension  \((\gK\He,\gV\He)\) qui est un corps valué hensélien tel que pour toute autre extension hensélienne \((\gL,\gW)\), il existe un unique morphisme de \((\gK\He,\gV\He)\) vers \((\gL,\gW)\). À priori, s'il existe, le hensélisé est unique à isomorphisme unique près. En outre, vu la \prt universelle du hensélisé d'un anneau local résiduellement discret, on a un unique morphisme local du hensélisé $\gV\he$ vers le hensélisé $\gV\He$. 

L'article \cite{KL00}  donne une \demo \cov de l'existence du hensélisé d'un corps valué. 

Celui-ci est construit par étapes. Une étape \elr consiste à ajouter un zéro à la Hensel de manière optimale (\cad que l'extension doit satisfaire la \prt universelle adéquate).

Avant même d'avoir construit \(\gK\He\), si \(\beta\) est le zéro à la Hensel d'un \pol de Nagata~\(f\), on peut construire \(\gK[\beta]\) en utilisant la méthode généralisée du \pgn. Si l'on note \(\gK[x]=\aqo{\KX}f\)
il s'agit en particulier de prolonger la valuation~$v$ à \(\gK[x]\) (les \elts $y$ pour lesquels \(v(y)=\infty\) sont annulés de force). En particulier,~\(\gK[\beta]\)  est un \cdi quotient de \(\gK[x]\), mais en l'absence d'\algo de \fcn des \pols sur \(\gK\) on ne sait pas déterminer le degré de \(\gK[\beta]\) sur \(\gK\). Cette situation est analogue à celle de la construction d'une clôture réelle d'un corps ordonné discret\footnote{Par exemple on peut calculer de manière sûre dans $\gK[\sqrt 2]$ même si on ne sait pas si $2$ admet une racine carrée dans le corps ordonné discret $\gK$.}.

\newcommand{\eal}{^{\mathrm{alg}}}
Dans la proposition suivante, on utilise une \textsl{clôture \agq hypothétique}
$\gK\eal$ de~\(\gK\), et une \textsl{extension hypothétique}  \((\gK\eal,\gV\eal)\)
du corps valué \((\gK,\gV)\). Cela n'affecte pas le caractère constructif de la construction de \((\gK\He,\gV\He)\), comme expliqué dans~\cite{KL00}. Précisément, on peut considérer \((\gK\eal,\gV\eal)\) comme une \textsl{\sad} qui ne s'effondre pas (voir \cite{CLR2001}), ce qui autorise à \gui{faire comme si} \((\gK\eal,\gV\eal)\) était un structure \agq usuelle pour conduire les calculs qui aboutissent à la construction du hensélisé.

\smallskip 
On dit qu'un \elt \(\alpha\in\gK\eal\) possède une \textsl{description immédiate} s'il s'écrit sous la forme \(a(1+\mu)\) où \(a\in\gK\) et \(\mu\in\fm_{\gV\eal}\).

\begin{proposition}[Lemme de Hensel, version polygone de Newton, \cite{KL00}] \label{lem NPH} 
 Soit $(\gK,\gV)$ un corps valué et $p\in \gV [X]$ un polynôme: 
$p(X)=\sum_{i=0,\ldots,d} p_iX^i$. Supposons que le \pgn de \(p\) 
admette une \gui{pente isolée} de $(k,v(p_k))$ à $(k+1,v(p_{k+1}))$. 
\begin{enumerate}
\item Alors l'unique zéro  $\alpha$ de \(p(X)\) dans \(\gK\eal\)  tel que 
$v(\alpha)=v(p_k)-v(p_{k+1})$ admet la \dimm suivante: 
\[\alpha = - { p_k  \over  p_{k+1} }\,(1+\mu) \quad \hbox{ avec } \; 
\mu \in \fm_{\gV\eal}
\]
\item En outre $\nu=1+\mu$  est un zéro à la Hensel codé par 
$(q,1)$ où  $q\in  \gV [Y]$ est donné par  
\[q(Y)={p_{k+1}^k \over p_k^{k+1}}\,p\Bigl(-{p_k\over p_{k+1}}Y\Bigr)\]
($\nu$ est l'unique zéro de $q$ dans $(\gV\eal)\eti$).
\item Si on pose \(\sum_{i=0,\ldots,d} r_iX^i:=r(X):=q(1+X)\), on a 
\(v(r_0)>0\) et  \(v(r_1)=0\). \\
Si $r_0=0$ alors $\mu=0$. Si $r_0\neq 0$ on pose \(s(X):=(1/r_0)r(-r_0X/r_1)\) et \(t(X)=X^ds(1/X)\) et on obtient que \(t(X)\) est un polynôme spécial dans \(\gV[X]\), de sorte que tous les zéros de $t(X)$, sauf le zéro spécial, qui correspond au zéro $\alpha$ de \(p(X)\), sont dans $\fm_{\gV^\mathrm{alg}}$. 
En résumé, un zéro $\alpha$ correspondant à une pente isolée d'un \pgn peut toujours être explicité soit comme un élément de \(\gK\), soit sous une forme 
\((a\beta+b)/(c\beta+d)\) où $\beta$  est le zéro spécial d'un 
polynôme spécial $t$, avec $a,b,c,d\in \gV$, \((c\beta+d)\neq 0$ et $(ad-
bc)\neq 0$.
\end{enumerate}
\end{proposition}

\begin{corollary} \label{corlem NPH}
Pour tout 
zéro à la Hensel $\alpha$, on a  \(\gK[\alpha]=\gK[\beta]\) où   $\beta$ 
est le zéro spécial d'un polynôme spécial. 
\end{corollary}

En effet, un zéro à la Hensel $\alpha$ codé par $(r,a)\in(\VX,\gV)$ 
correspond à une pente isolée, de  \((0,v(p_0))\)  à  \((1,v(p_1))\), 
du \pgn du \pol de Nagata $p(X)=r(X+a)$.
Le fait que la valuation de \(\gK\) s'étend de manière unique en une valuation de~\(\gK[\alpha]\), et que le groupe de valeurs et le corps résiduel ne changent pas, est \egmt démontré dans l'article.

Précisément, vu la proposition \ref{lem NPH}, la 
possibilité de calculer explicitement dans $(\gK\He,\gV\He)$ est ramenée 
à la proposition suivante.

\begin{proposition} 
\label{prop calc Hens} Soit \(t(X)\) un polynôme spécial, \(\beta\) son 
zéro spécial dans $\gK\He$, et~\hbox{\(q(X)\in \gK[X]\)} un \pol arbitraire, alors une \dimm de 
$q(\beta)$ peut être obtenue par des calculs uniformes dans $(\gK,\gV)$.
\end{proposition}

En conséquence la structure de  \(\gK[\beta]\)  comme corps valué est complètement déterminée, à un isomorphisme unique près.
Cette unicité forte permet donc de construire le hensélisé~\(\gK\He\) comme colimite filtrée des extensions \(\gK[\beta_1,\dots,\beta_m]\) construites en ajoutant des zéros spéciaux de \pols spéciaux successifs.

\section[Comparaison des premiers étages des deux hensélisés]{Comparaison des premiers étages de construction des deux hensélisés}\label{sec2}

On considère dans cette section la situation suivante.

\smallskip 
\begin{itemize}
\item \(\gA\) est intègre  local \dcd avec \(\gK\) pour corps de fractions, \(\gK\) est un \cdi;
\item \(\gV\) est un \adv de \(\gK\), il est \dcd, il contient \(\gA\) et il domine \(\gA\): \(\Ati=\gA\cap\gV\eti\) et \(\fm_\gA=\gA\cap\fm_\gV\);   
\item  \((\gK\eal,\gV\eal)\) est une extension hypothétique du corps valué \((\gK,\gV)\) avec \(\gK\eal\) une clôture \agq hypothétique de \(\gK\), et l'on considère la sous-extension  \((\gK\He,\gV\He)\) obtenue en ajoutant des zéros à la Hensel successifs; 
\item  \(t(X)\in\AX\) est un \pol spécial;%
\item   \(f(X):=t(X+1)\) est un \pol de Nagata; 
\item  \(\alpha\) est le zéro de \(f\) dans \(\fm_{\gA_f}\);
\item  \(\beta\) est l'unique zéro de \(f(X)\) dans \(\fm_{\gV\He}\), les autres zéros de \(f(X)\) dans \(\gK\eal\) sont dans \(-1+\fm_{\gV\eal}\subseteq (\gV\eal)\eti\);
\item  \(\gL:=\gK[\beta]\) et \((\gL,\gW)\) est l'extension de \((\gK,\gV)\) construite dans la proposition~\ref{lem NPH},
on~a~\(\gW=\gL\cap \gV\eal\);
\item  \(\theta_f:\gA_f\to\gW\) est le morphisme local qui factorise le morphisme \(\gA\to\gV\to\gW\), \hbox{on a \(\theta_f(\alpha)=\beta\)}. 
\end{itemize}

\smallskip On note que les structures \(\gA_f\), \(\gA\he\), \(\gW\) et \(\gV\He\) sont effectivement construites, tandis que les structures hypothétiques qui permettent de faciliter les \demos sont des \sads qui ne s'effondrent pas (voir  par exemple dans \cite{CLR2001} les 
théorèmes 1.1 et~4.3 et la remarque 4.6).

Dans le diagramme ci-dessous, tous les objets sont construits, les $\jmath$ sont des morphismes injectifs,
et $\pi$ est surjectif. On a aussi $\gK[x]=\aqo\KX f$,  $\gA[x]=\aqo\AX f$, \(\alpha=\varphi(x)\) \hbox{et \(\beta=\pi(x)\)}. On rappelle que \(\gA_f=U_f^{-1}\gA[x]\)

\[
\xymatrix @R=2em @C=3em{
&\gA \ar[ddd]_{\displaystyle \jmath_1} \ar[dl]_{\displaystyle \jmath_\gA} \ar[r]^{\displaystyle \jmath_4} & \gV\ar[r]^{\displaystyle \jmath_5}\ar[ddd]_{\displaystyle \jmath_2}& \gK\ar[ddd]_{\displaystyle \jmath_3}\ar[dr] ^{\displaystyle \jmath_\gK}
\\
\gA[x] \ar[rrrr]^{~~~~~~~~\displaystyle \jmath[x]} \ar[ddr]_{\displaystyle \varphi}&&&&\gK[x]\ar[ddl]^{\displaystyle \pi}
\\~
\\
&\gA_f\ar[r]_{\displaystyle \theta_f}&\gW\ar[r]_{\displaystyle \jmath_6}&\gK[\beta]
}
\]

Dans la suite, nous regardons \(\jmath_6\) comme un morphisme d'inclusion
et nous écrivons \(\jmath_6(\varepsilon)=\varepsilon\) pour \(\varepsilon\in\gW\).

Voici une version finitiste du \tho 1 de de Felipe et Teissier cité dans l'introduction. Rappelons qu'un \idep est dit \textsl{minimal} si l'anneau local correspondant est \zed. Dans le cas d'un anneau non trivial réduit, cela signifie que le localisé est un corps discret. 

\begin{theorem} \label{thFeTe1}
Le noyau de \(\theta_f\) est un \idemi de \(\gA_f\). 
\end{theorem}
%
\begin{proof}
Le noyau  \(\fp_f:=\Ker\theta_f\) est un \idep détachable   parce que~\(\gW\) est un anneau intègre, donc qui possède un test à $0$. 
Notons \(S_f=\gA_f\!\setminus\fp_f\).
Comme $\gA_f$ est réduit, l'\idep \(\fp_f\) est minimal \ssi le localisé \(S_f^{-1}\gA_f\) est un \cdi. Ce qui signifie que pour tout \(\gamma\in\gA_f\), ou bien \(\gamma\in S_f\), ou bien il existe un \(\zeta\in S_f\) tel que  \(\zeta\,\gamma=0\) dans \(\gA_f\).
Il suffit de le vérifier pour les \(\gamma\in\gA_f\) qui sont de la forme  \( \varphi(q(x))=q(\alpha) \) avec~\( q\in\AX \) car $\gA_f$ est un localisé de \(\gA[x]\). Enfin, pour \(\gamma=q(\alpha) \) on a 
\[ 
\theta_f(\gamma)=\pi(q(x))=q(\beta)=:\delta 
\] 
(notez que \(\gA[x]\subseteq \Kx\)). 
\\
Notons  \(\xi_1,\dots,\xi_n=\beta\)  les zéros de \(f\) dans \(\gK\eal\). Les valuations des \(\xi_i\) sont $w_1= \cdots =  w_{n-1}=0< w_n=v(f(0))$. 
Notons \(y=q(x)\in \gA[x]\subseteq \Kx\). 
\\
Rappelons comment se passe le test~à~\(0\) pour l'\elt~\(\pi(y)=\delta\) de~\(\gK[\beta]\). 
On calcule le \polcar \(g(T)=\chi_{y}(T)\) de  \(y\) dans la \Alg~\(\gA[x]\)  (qui est un module libre de rang \(n\) sur \(\gA\)). Par Cayley-Hamilton on a
 \(g(y)=0\), donc \(g(\delta)=g(q(\beta))=\pi(g(y))=0\) et
 \(g(\gamma)=g(q(\alpha))=\varphi(g(y))=0\).\\
Si \(g(0)\neq 0\), alors \(y\) est inversible dans \(\gK[x]\), donc \(\delta=\theta_f(\gamma)\) est inversible dans \(\gK[\beta]\) et non nul dans \(\gW\), i.e. \(\gamma\in S_f\).\\
 Si \(g(0)= 0\), alors \(g(T)=T^kh(T)\) avec \(h(0)\neq 0\) et 
 \(0<k\leq n\). 
Donc \(\gamma^k h(\gamma)=0\) dans~\(\gA_f\), d'où \(\gamma h(\gamma)=0\) car \(\gA_f\) est réduit. On a aussi \(\delta^k h(\delta)=0\) dans \(\gW\). Comme c'est un anneau intègre, \(\delta=0\) ou \(h(\delta)=0\).  
Cette disjonction est obtenue \cot selon l'\algo qui va suivre.
\\
Si \(h(\delta)=0\), on ne peut pas avoir \(\delta=0\) car sinon 
\(h(0)=h(\delta)=0\); donc \(\delta\neq 0\) et \(\gamma\in S_f\).
\\
Si \(\delta=0\), i.e. \(\gamma\in\fp_f\), comme \(h(0)\neq 0\) on a \(h(\delta)\neq 0\), et donc \(\zeta=h(\gamma)\in S_f\). On obtient \(\zeta\,\gamma=0\) dans \(\gA_f\) avec \(\zeta\in S_f\).
\\
En résumé, on a obtenu exactement ce qu'on voulait: \(\fp_f\) est un \idep minimal.
\\
Pour que la \demo \cov soit complète, rappelons comment on teste \hbox{si \(\delta=0\)} dans \(\gK[\beta]\). Il reste à traiter le cas où \(g(0)=0\)
\\
Le \pgn de \(g\) nous fournit dans le désordre les valuations $v_1\leq v_2\leq \cdots\leq v_n=\infty$ des \(q(\xi_i)\) dans  \((\gK\eal,\gV\eal)\).   
Pour savoir quelle valeur~$v_i$ correspond à \(\xi_n=\beta\), on considère l'\elt \(xq(x)=xy\) de \(\Kx\), on calcule son \polcar \(g_1(T)=\chi_{xy}(T)\) et on calcule le \pgn de~$g_1$. Ce dernier nous donne dans le désordre les valeurs des $\xi_iq(\xi_i)$. Si \(q(\beta)=0\), ce sont les mêmes valeurs $v_1\leq v_2\leq \cdots\leq v_n=\infty$ parce que \(\infty+w_n=\infty\). Et si \(q(\beta)\neq 0\), une et une seule des $v_i$ finies change, en augmentant de~$w_n$. 
\end{proof}

La \demo du \thref{thFT} de l'introduction peut être considérée comme terminée car on sait que toute extension \(\gK[\beta_1,\dots,\beta_m]\) obtenue en ajoutant des zéros à la Hensel successifs peut être obtenue en une seule étape. Ce résultat ne semble pas avoir de \demo \elr \cov. Il résulte par exemple du lemme de Hensel multivarié dont une \demo \cov est donnée dans
\cite{ACL2014}. 

Dans la section suivante, nous donnons une légère \gnn du \thref{thFeTe1}
qui règle la question sans faire appel au résultat difficile précédent.

\section{Une légère \gnn}

Nous remplaçons l'anneau local intègre \dcd dominé par un \adv, de la section précédente, par un anneau local \dcd réduit \textsl{minimalement valué}.

Cela signifie que nous donnons un groupe totalement ordonné discret \(\Gamma\), un anneau commutatif réduit \(\gA\) et une \textsl{valuation minimale} \(v:\gA\to\Gamma_\infty\), \cad une application  qui satisfait les axiomes suivants (les variables sont dans \(\gA\)).
\begin{enumerate}
\item \(v(a)\geq 0\).
\item \(v(ab)=v(a)+v(b)\).
\item \(v(a)=0\Leftrightarrow a\in\Ati\) et \(v(a)>0\Leftrightarrow a\in\Rad\gA\). 
\item \(v(a+b)\geq \min(v(a),v(b))\).
\item \(v(a)< \infty \hbox{ ou } \exists b\, (v(b)< \infty \hbox{ et } ba=0)\) \quad (minimalité).
\end{enumerate}

\noindent Dans ces conditions, on a:

\smallskip 
\begin{itemize}
\item \(\gA\) est un anneau local \dcd,
\item  le sous-ensemble \(\fp:=\sotq{a\in\gA}{v(a)=\infty}\) est un \idemi déta\-chable,
\item  $v$ définit sur le corps de fractions \(\gK\) de \(\gA':=\gA/\fp\) une valuation dont l'anneau \(\gV:=\sotq{x\in\gK}{v(x)\geq 0}\) domine \(\gA'\); on note \(\theta\) le morphisme naturel composé \(\gA\to\gA'\to\gK\).
\end{itemize}

\smallskip\noindent  On considère alors la situation suivante.

\smallskip 
\begin{itemize}
\item  \((\gK\eal,v)\) est une extension hypothétique du corps valué \((\gK,v)\) avec \(\gK\eal\) une clôture \agq hypothétique de \(\gK\), et l'on considère la sous-extension  \((\gK\He,v)\) obtenue en ajoutant des zéros à la Hensel successifs; 
\item  \(t\in\AX\) est un \pol spécial; $f(X)=t(X+1)$ est un \pol de Nagata; 
\item  \(\alpha\) est le zéro de \(f\) dans \(\fm_{\gA_f}\subseteq \fm_{\gA\he}\);
\item  \(\beta\) est le zéro de \(f\) dans \(\fm_{\gV\He}\), les autres zéros de \(f\) sont dans \((\gV\eal)\eti\);
\item  \((\gK[\beta],v)\)  est l'extension de \((\gK,v)\) construite dans la proposition~\ref{lem NPH};
\item  \(\theta_f:\gA_f\to\gK[\beta]\) est l'unique morphisme  de \Algs 
tel que \(\theta_f\circ \varphi=\pi\circ \theta[x]\), en particulier \(\theta_f(\alpha)=\beta\). 
\end{itemize}
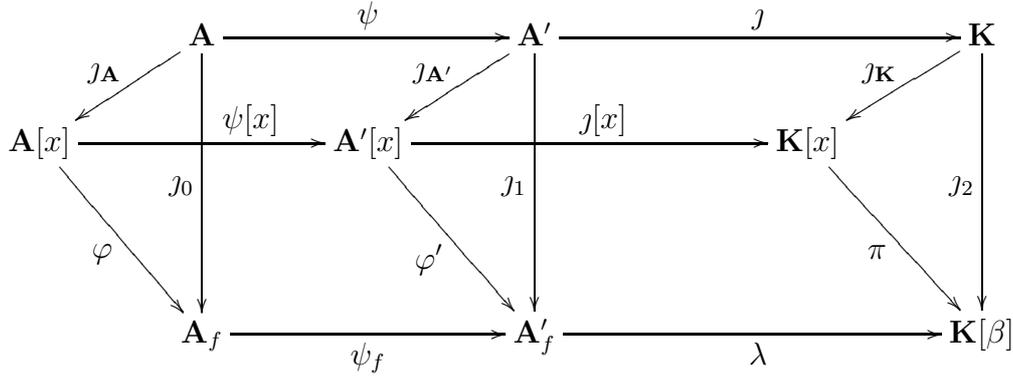
\vspace{-.2em}
\[
\xymatrix @R=2em @C=3em{
&\gA\ar[rr]^{\displaystyle \psi}\ar[ddd]_{\displaystyle \jmath_0} \ar[dl]_{\displaystyle \jmath_\gA}&&\gA' \ar[ddd]_{\displaystyle \jmath_1} \ar[dl]_{\displaystyle \jmath_{\gA'}} \ar[rrr]^{\displaystyle \jmath} && 
& \gK\ar[ddd]_{\displaystyle \jmath_2}\ar[dl] _{\displaystyle \jmath_\gK}
\\
\gA[x]\ar[rr]^{~~~~~~~~~\displaystyle \psi[x]}\ar[ddr]_{\displaystyle \varphi}&&\gA'[x] \ar[rrr]^{~~~\displaystyle \jmath[x]} \ar[ddr]_{\displaystyle \varphi'}&&&\gK[x]\ar[ddr]_{\displaystyle \pi}
\\~\\
&\gA_f\ar[rr]_{\displaystyle \psi_f}&&\gA'_f\ar[rrr]_{\displaystyle {\lambda}}&&&\gK[\beta]
}
\]

La partie droite du diagramme ci-dessus correspond à la situation examinée dans la section précédente, en y remplaçant \(\gA\) par \(\gA'\).

La situation est plus simple avec le diagramme ci-dessous, extrait du précédent.
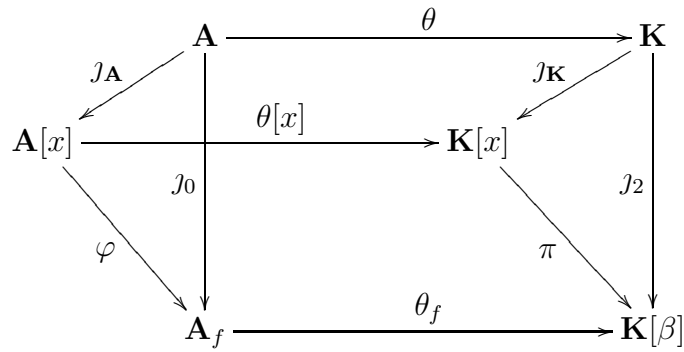
\vspace{-.1em}
\[
\xymatrix @R=2em @C=3em{
&\gA\ar[rrr]^{\displaystyle \theta}\ar[ddd]_{\displaystyle \jmath_0} \ar[dl]_{\displaystyle \jmath_\gA}&&& \gK\ar[ddd]_{\displaystyle \jmath_2}\ar[dl] _{\displaystyle \jmath_\gK}
\\
\gA[x]\ar[rrr]^{~~~~\displaystyle \theta[x]}\ar[ddr]_{\displaystyle \varphi}&&&\gK[x]\ar[ddr]_{\displaystyle \pi}
\\
\\
&\gA_f\ar[rrr]^{\displaystyle \theta_f}&&&\gK[\beta]
}
\]
On a  $\gK[x]=\aqo\KX f$,  $\gA[x]=\aqo\AX f$, \(\alpha=\varphi(x)\), \(\beta=\pi(x)\) et \(\theta_f(\alpha)=\beta\).
 
\begin{theorem} \label{thFeTe2}
Le noyau de \(\theta_f\) est un \idemi de \(\gA_f\). 
\end{theorem}
%
\begin{proof}
Le noyau  \(\fp_f:=\Ker\theta_f\) est un \idep détachable   parce que~\(\gK[\beta]\) est un \cdi. 
Notons \[S=\gA\setminus\fp=\sotq{c\in\gA}{\theta(c)\neq 0}
\;\hbox{ et }\;
S_f=\gA_f\setminus\fp_f=\sotq{\gamma\in\gA_f}{\theta_f(\gamma)\neq 0}.\]
Comme \(\gA\) est réduit, le fait que l'\idep~\(\fp\) est minimal signifie que pour tout \(a\in\gA\),  \(a\in S\) ou il existe un \(b\in S\) tel que \(ba=0\).
De même, comme $\gA_f$ est réduit, l'\idep~\(\fp_f\) est minimal \ssi  pour tout \(\gamma\in\gA_f\), ou bien \(\gamma\in S_f\), ou bien il existe un~\(\zeta\in S_f\) tel que  \(\zeta\,\gamma=0\).
Pour démontrer cette disjonction, il suffit de le faire pour les \(\gamma\in\gA_f\) qui sont de la forme  \(\varphi(q(x))=q(\alpha) \) avec~\( q\in\AX \) car~$\gA_f$ est un localisé de~\(\gA[x]\). Notons \(q_1=\theta[X](q)\).
On  a alors 
\[ \theta_f(\gamma)=\theta_f(q(\alpha))=\pi(q_1(x))=q_1(\beta)=:\delta \] 
Notons \(y=q(x)\in \gA[x]\) et \(z=\theta[x](y)=q_1(x)\in \Kx\).
On calcule le \polcar \(g(T)=\chi_{y}(T)\) de  \(y\) dans la \Alg~\(\gA[x]\)  (qui est un module libre de rang \(n\) sur \(\gA\)). 
Le même calcul pour \(z\in\Kx\) donne le \polcar
\(g_1(T)=\chi_{z}(T)\in\gK[T]\) avec \(g_1=\theta[T](g)\):
\[g(T)=T^n+\som_{j=0}^{n-1}a_jT^j, \quad  g_1(T)=T^n+\som_{j=0}^{n-1}\theta(a_j)T^j.\] 
On a \(g(y)=0\) dans \(\gA[x]\) et \(g_1(z)=0\) dans \(\gK[x]\),
donc  \(g(\gamma)=0\) dans \(\gA_f\) et \(g_1(\delta)=0\) dans \(\gK[\beta]\). On a \(g_1(0)= 0\) ou \(g_1(0)\neq  0\) car \(\gK\) est discret; et \(\delta=0\) ou \(\delta\neq 0\) car  \(\gK[\beta]\) est discret.
\\
Supposons d'abord que  \(g_1(0)\neq 0\), alors \(\delta\neq  0\) car \(\delta= 0\) implique \(g_1(0)= 0\). Donc \(\gamma\in S_f\).  
\\
Dans le cas contraire, on écrit \(g_1(T)=T^kh_1(T)\) avec \(h_1(0)\neq 0\) dans~\(\gK\) \hbox{et  \(0< k\leq n\)}. 
On écrit \(g(T)=a(T)+T^kh(T)\) avec \(\deg_T(a)<k\).  On~a 
\(\theta[T](a)=0\) et \(\theta[T](h)=h_1\).
On~a donc \(\theta(a_j)=0\) pour \(0\leq j<k\). 
Comme \(\theta(a_j)\neq 0\) est impossible, il existe \(b_j\in S\) tel que \(b_ja_j=0\) dans \(\gA\). 
Soit \(b\) le produit des \(b_j\) pour \(0\leq j<k\). On a \(b\in S\), \(ba(T)=0\) et \(bg(T)=T^kbh(T)\).
  Comme \(\delta^kh_1(\delta)=g_1(\delta)=0\), \hbox{on a}  \(\delta=0\) ou \(h_1(\delta)=0\).\\
Si \(h_1(\delta)=0\), on ne peut pas avoir \(\delta=0\) car sinon 
\(h_1(0)=h_1(\delta)=0\). Donc \(\delta\neq 0\), \hbox{i.e. \(\gamma\in S_f\)}.
\\
Si \(\delta=0\),  on a \(\theta_f(h(\gamma))=h_1(\delta)=h_1(0)\neq 0\), et donc \(\zeta=h(\gamma)\in S_f\). 
Puisque \(bg(\alpha)=0\) et \(bg=bT^kh\) on obtient \(b\gamma^k \zeta=0\) dans \(\gA_f\) 
avec \(b\zeta\in S_f\). Enfin, comme \(\gA_f\) est réduit, \hbox{on a \(b\zeta\,\gamma=0\)} 
dans \(\gA_f\).
\\
En résumé, on a obtenu exactement ce qu'on voulait: \(\fp_f\) est un \idep minimal.
\end{proof}

\subsection*{Conclusion}

En utilisant de manière répétée le \thref{thFeTe2} et en démarrant avec la situation de la section \ref{sec2}, on obtient à tous les étages finis de la construction simultanée de \(\gA\he\) et \(\gV\He\) le résultat donné par le \thref{thFT} de de Felipe et Teissier.

\section{Annexe: une autre \gnn}


\begin{theorem} \label{thFeTe3}
Avec les mêmes hypothèses qu'au \thref{thFeTe2}, mais sans supposer \(\gA\) réduit, on a la même conclusion. 
\end{theorem}
\begin{proof}
Le noyau  \(\fp_f:=\Ker\theta_f\) est un \idep détachable   parce que~\(\gK[\beta]\) est un \cdi. 
Notons \[S=\gA\setminus\fp=\sotq{c\in\gA}{\theta(c)\neq 0}
\;\hbox{ et }\;
S_f=\gA_f\setminus\fp_f=\sotq{\gamma\in\gA_f}{\theta_f(\gamma)\neq 0}.\]
Le fait que l'\idep \(\fp\) est minimal signifie que pour tout \(a\in\gA\),  \(a\in S\) ou il existe un \(b\in S\) tel que \(ba\) est nilpotent.
De même, l'\idep \(\fp_f\) est minimal \ssi  pour tout \(\gamma\in\gA_f\), ou bien \(\gamma\in S_f\), ou bien il existe un \(\zeta\in S_f\) tel que  \(\zeta\,\gamma\) est nilpotent.
Pour démontrer cette disjonction, il suffit de le faire pour les \(\gamma\in\gA_f\) qui sont de la forme  \(\varphi(q(x))=q(\alpha) \) avec~\( q\in\AX \) car~$\gA_f$ est un localisé de~\(\gA[x]\). Notons \(q_1=\theta[X](q)\).
On  a alors 
\[ \theta_f(\gamma)=\theta_f(q(\alpha))=\pi(q_1(x))=q_1(\beta)=:\delta \] 
Notons \(y=q(x)\in \gA[x]\) et \(z=\theta[x](y)=q_1(x)\in \Kx\).
On calcule le \polcar \(g(T)=\chi_{y}(T)\) de  \(y\) dans la \Alg~\(\gA[x]\)  (qui est un module libre de rang \(n\) sur \(\gA\)). 
Le même calcul pour \(z\in\Kx\) donne le \polcar
\(g_1(T)=\chi_{z}(T)\in\gK[T]\) avec \(g_1=\theta[T](g)\):
\[g(T)=T^n+\som_{j=0}^{n-1}a_jT^j, \quad  g_1(T)=T^n+\som_{j=0}^{n-1}\theta(a_j)T^j.\] 
On a \(g(y)=0\) dans \(\gA[x]\) et \(g_1(z)=0\) dans \(\gK[x]\),
donc  \(g(\gamma)=0\) dans \(\gA_f\) et \(g_1(\delta)=0\) dans \(\gK[\beta]\). On a \(g_1(0)= 0\) ou \(g_1(0)\neq  0\) car \(\gK\) est discret; et \(\delta=0\) ou \(\delta\neq 0\) car  \(\gK[\beta]\) est discret.
\\
Supposons d'abord que  \(g_1(0)\neq 0\), alors \(\delta\neq  0\) car \(\delta= 0\) implique \(g_1(0)= 0\). Donc \(\gamma\in S_f\).  
\\
Dans le cas contraire, on écrit \(g_1(T)=T^kh_1(T)\) avec \(h_1(0)\neq 0\) dans~\(\gK\) \hbox{et  \(0< k\leq n\)}. On écrit \(g(T)=a(T)+T^kh(T)\) avec \(\deg_T(a)<k\). On a 
\(\theta[T](a)=0\) et \(\theta[T](h)=h_1\). On~a donc \(\theta(a_j)=0\) pour \(0\leq j<k\), et \(\theta[T](h)=h_1\). 
Comme \(\theta(a_j)\neq 0\) est impossible, il existe \(b_j\in S\) tel que \(b_ja_j\) est nilpotent dans \(\gA\). 
Soit \(b\) le produit des \(b_j\) pour \(0\leq j<k\). On a \(b\in S\), et, pour un entier \(N\) assez grand, \(ba(T)^N=0\). 
\\
Comme \(\delta^kh_1(\delta)=g_1(\delta)=0\), \hbox{on a}  \(\delta=0\) ou \(h_1(\delta)=0\).\\
Si \(h_1(\delta)=0\), on ne peut pas avoir \(\delta=0\) car sinon 
\(h_1(0)=h_1(\delta)=0\). Donc \(\delta\neq 0\), \hbox{i.e. \(\gamma\in S_f\)}.
\\
Si \(\delta=0\),  on a \(\theta_f(h(\gamma))=h_1(\delta)=h_1(0)\neq 0\), et donc \(\zeta=h(\gamma)\in S_f\). 
Puisque \(bg(\alpha)=0\) et \(bg(T)=ba(T)+bT^{k}h\) on obtient \(ba(\gamma)+b\zeta\,\gamma^{k} =0\) dans \(\gA_f\). 
D'où, puisque \(ba(T)^N=0\),
 \[0=(-ba(\gamma))^N=b^N\zeta^N\,\gamma^{Nk} =0\] avec \(b^N\zeta^N\in S_f\).
\\
En résumé, on a obtenu exactement ce qu'on voulait: \(\fp_f\) est un \idep minimal.
\end{proof}

\addcontentsline{toc}{section}{Références}

\small
\bibliographystyle{plain-fr}
\bibliography{AlrecobibFr}

\end{document}

%% file: FrenchTheoremsBT.tex
\newtheorem{theorem}{Théorème}
\newtheorem{thm}[theorem]{Theorem}

\newtheorem{lemma}[theorem]{Lemme}
\newtheorem{corollary}[theorem]{Corolaire}
\newtheorem{proposition}[theorem]{Proposition}

\theoremstyle{definition}
\newtheorem{definition}[theorem]{Définition}

\theoremstyle{remark}

%% file: FrenchMacrosBT.tex
\newcounter{MF}
\newcommand\stMF{\stepcounter{MF}}

\newcommand{\lec}{\stMF\ifodd\value{MF}lecteur\xspace\else 
lectrice\xspace\fi}

\newcommand{\lecs}{\stMF\ifodd\value{MF}lecteurs\xspace\else 
lectrices\xspace\fi}

\newcommand{\alec}{\stMF\ifodd\value{MF}au lecteur\xspace\else%
à la lectrice\xspace\fi}

\newcommand{\dlec}{\stMF\ifodd\value{MF}du lecteur\xspace\else%
de la lectrice\xspace\fi}

\newcommand{\llec}{\stMF\ifodd\value{MF}le lecteur\xspace\else la lectrice\xspace\fi}

\newcommand{\Llec}{\stMF\ifodd\value{MF}Le lecteur\xspace\else La lectrice\xspace\fi}

\newcommand{\lui}{\ifodd\value{MF}lui\xspace\else
elle\xspace\fi}

\newcommand{\celui}{\ifodd\value{MF}celui\xspace\else
celle\xspace\fi}

\newcommand{\ceux}{\ifodd\value{MF}ceux\xspace\else
celles\xspace\fi}

\newcommand{\er}{\ifodd\value{MF}er\xspace\else
ère\xspace\fi}

\newcommand{\eux}{\ifodd\value{MF}eux\xspace\else
elles\xspace\fi}

\newcommand{\eUx}{\ifodd\value{MF}eux\xspace\else
euse\xspace\fi}

\newcommand{\eUX}{\ifodd\value{MF}eux\xspace\else
euses\xspace\fi}

\newcommand{\leux}{\ifodd\value{MF}leux\xspace\else
leuse\xspace\fi}

\newcommand{\il}{\ifodd\value{MF}il\xspace\else
elle\xspace\fi}

\newcommand{\ien}{\ifodd\value{MF}ien\xspace\else
ienne\xspace\fi}

\newcommand{\e}{\ifodd\value{MF}\xspace \else e\xspace\fi}

\newcommand{\n}{\ifodd\value{MF}n\xspace\else nne\xspace\fi}

\makeatletter
\newcommand{\la}{\@ifstar{\ifodd\value{MF}le\else
la\fi}{\stMF\ifodd\value{MF}le\else la\fi}}
\makeatother

\newcommand \thref[1] {théorème~\ref{#1}}

\newcommand\oge{\leavevmode\raise.3ex\hbox{$\scriptscriptstyle\langle\!\langle\,$}}
\newcommand\feg{\leavevmode\raise.3ex\hbox{$\scriptscriptstyle\,\rangle\!\rangle$}}

\newcommand\gui[1]{\oge{#1}\feg}

\renewcommand\paragraph[1]{

\medskip \noindent $\bullet$ \textbf{#1}}

\newcommand \cad {c'est-à-dire\xspace}

\newcommand \ssi {si, et seu\-lement si, }

\newcommand \Alg {$\gA$-\alg}
\newcommand \Algs {$\gA$-\algs}

\newcommand \adv {anneau de valuation\xspace}

\newcommand \afr {anneau \frl}

\newcommand \afrvr {\afr avec \ravs}

\newcommand \agq{algé\-bri\-que\xspace}
\newcommand \agqs{algé\-bri\-ques\xspace}

\newcommand \alg {algè\-bre\xspace}
\newcommand \algs {algè\-bres\xspace}

\newcommand \algo{algo\-rithme\xspace}

\newcommand \algq{al\-go\-rith\-mi\-que\xspace}

\newcommand \alo {an\-neau lo\-cal\xspace}

\newcommand \arc {anneau réel clos\xspace}

\newcommand \cdi{corps discret\xspace}

\newcommand \dcd {rési\-duel\-lement dis\-cret\xspace}

\newcommand \dimm {description immédiate\xspace}

\newcommand \demo{démon\-stra\-tion\xspace}     
\newcommand \demos{démon\-stra\-tions\xspace}

\newcommand \dfn{défi\-nition\xspace}

\newcommand \dvz {di\-viseur de zéro\xspace}

\newcommand \egmt {éga\-lement\xspace}

\newcommand \egt {éga\-li\-té\xspace}

\newcommand \elr{élé\-men\-taire\xspace}

\newcommand \elt{élé\-ment\xspace}  
\newcommand \elts{élé\-ments\xspace}

\newcommand \eqve {équi\-va\-lente\xspace}  
  
\newcommand \eqves {équi\-val\-entes\xspace}

\newcommand \esid {essen\-tiel\-lement iden\-tique\xspace}  
\newcommand \esids {essen\-tiel\-lement iden\-tiques\xspace}

\newcommand \eseq {essen\-tiel\-lement \eqve}  
\newcommand \eseqs {essen\-tiel\-lement \eqves}

\newcommand \fit {fidè\-lement\xspace}

\newcommand \fcn {factorisation\xspace}

\newcommand \fpte {\fit plate\xspace}

\newcommand \frl {for\-tement réticulé\xspace}

\newcommand \frls {for\-tement réticulés\xspace}

\newcommand\gnlt{géné\-ra\-lement\xspace}  

\newcommand\gnn{géné\-ra\-li\-sa\-tion\xspace}

\newcommand \grl{groupe \rtl}
\newcommand \grls{groupes \rtls}

\newcommand \icl {inté\-gra\-lement clos\xspace}

\newcommand \idema {idéal maxi\-mal\xspace}

\newcommand \idep {idéal pre\-mier\xspace}

\newcommand \idemi {\idep minimal\xspace}

\newcommand \idm {idem\-po\-tent\xspace}

\newcommand \idp {idéal prin\-ci\-pal\xspace}

\newcommand \inteq {intui\-ti\-vement \eqve}
\newcommand \inteqs {intui\-ti\-vement \eqves}

\newcommand \ivmp {il en va de même pour\xspace}

\newcommand \mo {mo\-no\"{\i}de\xspace}

\newcommand \ncr{néces\-saire\xspace}       
\newcommand \ncrs{néces\-saires\xspace}

\newcommand \pgn {polygone de Newton\xspace}

\newcommand \pol {poly\-nôme\xspace}
\newcommand \pols {poly\-nômes\xspace}

\newcommand \polcar {\pol carac\-té\-ris\-tique\xspace}

\newcommand \prt {pro\-pri\-été\xspace}

\newcommand \ravs {racines virtuelles\xspace}

\newcommand \rtl {réti\-culé\xspace}
\newcommand \rtls {réti\-culés\xspace}

\newcommand \sad {\salg dynamique\xspace}
\newcommand \sads {\salgs dynamiques\xspace}

\newcommand \salg {structure \agq}
\newcommand \salgs {structures \agqs}

\newcommand \sdz {sans \dvz}

\newcommand \tho {théo\-rème\xspace}
\newcommand \thos {théo\-rèmes\xspace}

\newcommand \uvle {uni\-ver\-selle\xspace}

\newcommand \zed {z\'{e}ro-di\-men\-sionnel\xspace}

\newcommand \zedr {\zed réduit\xspace}

\newcommand \cof {cons\-truc\-tif\xspace}
\newcommand \cofs {cons\-truc\-tifs\xspace}

\newcommand \cov {cons\-truc\-tive\xspace}
\newcommand \covs {cons\-truc\-tives\xspace}

\newcommand \coma {\maths\covs}
\newcommand \clama {\maths clas\-siques\xspace}

\renewcommand \cot {cons\-truc\-ti\-vement\xspace}

\newcommand \maths {mathé\-ma\-tiques\xspace}

%% file: MathMacrosBT.tex
\newcommand {\junk}[1]{}

\renewcommand \leq{\leqslant}

\renewcommand \geq{\geqslant}

\newcommand\eti{^\times}

\newcommand \etoz{$^*$}

\newcommand \Ati {\gA^{\!\times}}

\newcommand \aqo[2] {#1\sur{\gen{#2}}\!}

\newcommand \gen[1] {\left\langle{#1}\right\rangle}

\newcommand \so[1] {\left\{\,{#1}\, \right\}}
\newcommand \sO[1]{\big\{{#1}\big\}}
\newcommand \sotq[2]{\so{#1\mathrel{;}#2}}
\newcommand \sotQ[2]{\sO{#1\mathrel{;}#2}}
\newcommand \sur[1] {\!\left/#1\right.}

\renewcommand \leq{\leqslant}
\renewcommand \geq{\geqslant}

\newcommand \som {\sum\nolimits}

\newcommand\tsbf[1]{\textsf{\textbf{\textup{#1}}}}

\newcommand\Tsbf[1]{\hyperref[Ax#1]{\tsbf{#1}}}

\newcommand\Sa[1]{\hyperref[theorie#1]{\sa{#1}}}
\newcommand\sa[1]{\hbox{\usefont{T1}{pzc}{m}{it}#1}\,}

\newcommand \ov[1] {\overline{#1}}

\makeatletter
\def\revddots{\mathinner{\mkern1mu\raise\p@
\vbox{\kern7\p@\hbox{.}}\mkern2mu
\raise4\p@\hbox{.}\mkern2mu\raise7\p@\hbox{.}\mkern1mu}}
\makeatother

\newcommand \gA {\mathbf{A}}
\newcommand \gB {\mathbf{B}}
\newcommand \gC {\mathbf{C}}

\newcommand \gK {\mathbf{K}}

\newcommand \gL {\mathbf{L}}

\newcommand \gV {\mathbf{V}}

\newcommand \gW {\mathbf{W}}

\newdimen\xyrowsp
\newcommand{\SCO}[6]{
\xymatrix @R = \xyrowsp {
                                  &1 \ar@{-}[dl] \ar@{-}[dr] \\
#3 \ar@{-}[ddr]                   &   & #6 \ar@{-}[ddl] \\
                                  &\bullet\ar@{-}[d] \\
                                  &\bullet   \\
#2 \ar@{-}[ddr] \ar@{-}[uur]      &   & #5 \ar@{-}[ddl] \ar@{-}[uul] \\
                                  &\bullet \ar@{-}[d] \\
                                  &\bullet  \\
#1 \ar@{-}[uur]                   &   & #4 \ar@{-}[uul] \\
                                  & 0 \ar@{-}[ul] \ar@{-}[ur] \\
}
}

\renewcommand \deg {\MA{\mathrm{deg}}}

\newcommand \Frac {\MA{\mathrm{Frac}}}

\newcommand \Ker {\MA{\mathrm{Ker}}}

\newcommand \Rad {\MA{\mathrm{Rad}}}

\newcommand\MA[1]{\mathop{#1}\nolimits}

\newcommand\fm{\mathfrak{m}}

\newcommand\fp{\mathfrak{p}}

\newcommand \AX {\gA[X]}

\newcommand \KX {\gK[X]}
\newcommand \Kx {\gK[x]}
\newcommand \VX {\gV[X]}